\documentclass[11pt]{article}

\usepackage[margin=1in,a4paper]{geometry}
\usepackage{amssymb,amsmath}
\usepackage{amsthm} 
\usepackage{graphicx} 
\usepackage{xcolor}
\usepackage{esint}
\usepackage[colorlinks,linkcolor=blue, urlcolor  = blue,citecolor=blue,pagebackref,hypertexnames=false, breaklinks]{hyperref}

\usepackage[nottoc,notlot,notlof]{tocbibind}

\def\Xint#1{\mathchoice
    {\XXint\displaystyle\textstyle{#1}}%
    {\XXint\textstyle\scriptstyle{#1}}%
    {\XXint\scriptstyle\scriptscriptstyle{#1}}%
    {\XXint\scriptscriptstyle\scriptscriptstyle{#1}}%
    \!\int}
    \def\XXint#1#2#3{{\setbox0=\hbox{$#1{#2#3}{\int}$}
    \vcenter{\hbox{$#2#3$}}\kern-.5\wd0}}
    \def\fint{\Xint-}

\newcommand{\vast}{\bBigg@{4}}
\newcommand{\Vast}{\bBigg@{5}}
\makeatother

\newcommand{\D}{\mathcal{D}}
\newcommand{\E}{\mathcal{E}}
\newcommand{\F}{\mathcal{F}}

\newcommand{\texton}{\text{ on }}

\theoremstyle{plain}
\newtheorem{thm}{Theorem}[section]
\newtheorem{lem}[thm]{Lemma}
\newtheorem{cor}[thm]{Corollary}
\newtheorem{prop}[thm]{Proposition}

\newtheorem{example}[thm]{Example}

\theoremstyle{definition}
\newtheorem{defn}[thm]{Definition}

\theoremstyle{remark}
\newtheorem{remark}[thm]{Remark}
\newcommand{\bremark}{\begin{remark} \em}
\newcommand{\eremark}{\end{remark} }
\usepackage{color}

\hbadness=\maxdimen

\begin{document}


\title{Dirichlet  fractional Gaussian fields on  the Sierpinski gasket  and their discrete graph approximations}

\author{Fabrice Baudoin\footnote{Partly supported by the NSF grant DMS~1901315.}, Li Chen\footnote{Partly supported by a Simons Foundation Travel Support for Mathematicians Grant \#853249.}}

\maketitle

\begin{abstract}
We define and study the Dirichlet fractional Gaussian fields on the  Sierpinski gasket and show that they are limits of fractional discrete Gaussian fields defined on the sequence of canonical approximating graphs.
\end{abstract}

\tableofcontents

\section{Introduction}

Let $K \subset \mathbb R^2$ be the Sierpinski gasket fractal.
A first goal of the paper is to introduce and study a family of Gaussian fields on $K$ indexed by a parameter $s \ge 0$ and satisfying
\begin{align}\label{intro}
\mathbb{E}( X_s(f) X_s(g))= \int_{K} (-\Delta)^{-s} f (-\Delta)^{-s} g d\mu,
\end{align}
where $f,g$ belong to a space of suitable test functions on $K$, $\mu$ is the Hausdorff measure and $\Delta$ is the Dirichlet Laplacian on $K$. Such a field is heuristically defined as the distribution $X_s=(-\Delta)^{-s}W$ where $W$ is a white noise on $L^2(K,\mu)$. We will mostly be interested in the regularity properties of those fields and in the convergence of their natural discretizations. Concerning the regularity properties, the value $s=\frac{d_h}{2d_w}$ is a critical value, where $d_h=\frac{\ln 3}{\ln 2}$ is the Hausdorff dimension of $K$ and  $d_w=\frac{\ln 5}{\ln 2}$  is called the walk dimension. More precisely, the study of $X_s$ is divided according to two ranges:

\begin{itemize}
\item $0 \le s \le \frac{d_h}{2d_w}$: For this range of parameters we show that the Gaussian field $X_s$ can not be defined pointwise but belongs to a Sobolev space of distributions that we identify;
\item $ s > \frac{d_h}{2d_w}$: For this range,  we show that the Gaussian field $X_s$ can be defined pointwise and admits a H\"older regular version.
\end{itemize}

The critical value $s=\frac{d_h}{2d_w}$ corresponds to a log-correlated field on $K$ that will tentatively be further studied in a later work. 

A second goal of the paper is to introduce discrete analogues of the fractional Gaussian fields $X_s$ by using the canonical graph approximation $G_m$, $m \ge 0$ of the Sierpinski gasket, see Figure \ref{figure1}.

 \begin{figure}[htb]\label{figure1}
  	\noindent
  \makebox[\textwidth]{\includegraphics[height=0.3\textwidth] {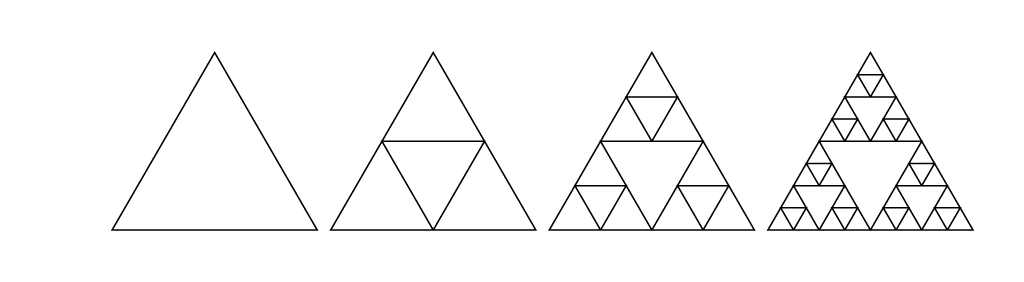}}
  	\caption{Sierpinski gasket graphs $G_0$, $G_1$, $G_2$ and $G_3$}
  \end{figure}
  
  
  Those discrete fields are  given by $X_s^m=(-\Delta_m)^{-s} W_m$ where $\Delta_m$ is the Dirichlet graph Laplacian on $G_m$, and $W_m$ is a sequence of i.i.d. standard Gaussian normal on the vertices of $G_m$. We will show the convergence of $X_s^m$ to $X_s$, first in law in a space of tempered distributions, and then in law in a suitable  Sobolev space.

\

This paper is natural complement to the recent paper \cite{BaudoinLacaux} which studied fractional Gaussian fields associated to the Neumann Laplacian on the Sierpinski gasket for the range of parameters  $  \frac{d_h}{2d_w} < s  <1- \frac{d_h}{2d_w}$. However, as noted above, in the present paper we are rather interested in the fractional Gaussian fields associated with the Dirichlet Laplacian, study the whole range of parameters $s \ge 0$ and also introduce the family of discrete fields $X_s^m$ for which we prove convergence when $m \to +\infty$. In a subsequent work, we plan to study the maxima  of the discrete log-correlated field on the gasket and their  possible rescaling limits; we refer for instance to  \cite{biskup} for an introduction and motivation to such questions.

\

The paper is organized as follows. In Section 2, after some preliminaries, we introduce the discrete and continuous fractional Gaussian fields on the gasket. A highlight result is Theorem \ref{existence FGF} which states the existence of a Gaussian random variable $X_s$ which takes values in a suitable space of tempered distributions and that satisfies \eqref{intro}. We then prove in Proposition \ref{density field} that this random tempered distribution $X_s$ defines an $L^2$ function on $K$ if and only if  $ s > \frac{d_h}{2d_w}$. Section 3 deals with the regularity theory of the random tempered distribution $X_s$. For $ s \le \frac{d_h}{2d_w}$ we quantify this regularity by introducing a scale of distributional Sobolev spaces and for $s >\frac{d_h}{2d_w}$, using the entropy method as in \cite{AdlerTaylor}, we study the H\"older regularity property of the $L^2$ function on $K$ defined by $X_s$. Finally, in Section 4, we prove the convergence of the discrete fields $X^m_s$ to $X_s$.

\

\textbf{Acknowledgment:} \textit{The authors thank an anonymous referee for the careful reading of an earlier version of the manuscript which led to  improvements in the presentation and arguments.}

\section{Discrete and continuous FGFs on the Sierpinski gasket}

\subsection{Discrete and continuous Dirichlet Laplacians}

We first define the (Dirichlet) Laplacians on the Sierpinski gasket and Sierpinski gasket graphs. For further references and more general fractals, see for instance \cite{Barlow, FukushimaShima, Kigami}.

Let $V_0=\{p_1, p_2, p_3\}$ be a set of vertices of an equilateral triangle of side 1 in $\mathbb C$. Define 
$$\mathfrak f_i(z)=\frac{z-p_i}{2}+p_i$$ for $i=1,2,3$. Then the Sierpinski gasket $K$ is the unique non-empty compact subset in $\mathbb C$ such that 
\[
K=\bigcup_{i=1}^3 \mathfrak f_i(K).
\]
The set $V_0$ is called the boundary of $K$, we will also denote it by $\partial K$.

The Hausdorff dimension of $K$ with respect to the Euclidean metric (denoted $d(x,y)=| x - y |$ in this paper) is given by $d_h=\frac{\ln 3}{\ln 2}$. A (normalized) Hausdorff measure on $K$ is given by the Borel measure $\mu$ on $K$ such that for any $i_1, \cdots, i_n \in \{ 1,2,3 \} $,
\[
\mu \left(  \mathfrak f_{i_1} \circ \cdots \circ \mathfrak f_{i_n}  (K)\right)=3^{-n}.
\]

%
%

This measure $\mu$ is $d_h$-Ahlfors regular, i.e. there exist constants $c,C>0$ such that for every $x \in K$ and $r \in [0, \mathrm{diam} (K) ]$,
\begin{equation}
\label{Ahlfors}
c r^{d_h} \le \mu (B(x,r)) \le C r^{d_h}.
\end{equation}
%
%
We define a sequence of sets $\{V_m\}_{m\ge 0}$ inductively by
\[
V_{m+1}=\bigcup_{i=1}^3 \mathfrak f_i(V_m).
\]
Then we have a natural sequence of Sierpinski gasket graphs (or pre-gaskets) $\{G_m\}_{m\ge 0}$ whose edges have length $2^{-m}$ and whose set of vertices is $V_m$, see Figure \ref{figure1}. Notice that $\#V_m=\frac{3(3^m+1)}2$. We will use the notations $V_*=\cup_{m\ge 0}V_m$ and $V_*^0=\cup_{m\ge 0}V_m\setminus V_0$.

For any $p\in V_m$, denote by $V_{m,p}$ the collection of neighbors of $p$ in $G_m$. Then $\#V_{m,p}=4$ if $p\notin V_0$ and $\#V_{m,p}=2$ if $p\in V_0$. Let $\ell(V_m)$ be the set of functions $f: V_m\to \mathbb R$. Then for any $f\in \ell(V_m)$, we consider the discrete Laplacian on $V_m$ defined by
\[
\Delta_m f(p)=5^m \sum_{q\in V_{m,p}} (f(q)-f(p)), \quad p \in V_m \setminus V_0.
\]
The semigroup generated by the discrete Laplacian $\Delta_m$ on $V_m$ is denoted by $\{P_t^m\}_{t\ge0}$.

Let $C(K)$ be the set of continuous functions on $K$. We define
\[
\mathcal D=\{f\in C(K), \text{ there exists }g\in C(K) \text{ such that }\lim_{m\to \infty}\max_{p\in V_m\setminus V_0} |\Delta_m f(p)-g(p)|=0\}.
\]
For $f\in \mathcal D$, the Kigami Laplacian $\Delta$ of $f$ on $K$ is then defined by 
\begin{equation}\label{eq:Laplacian}
\Delta f(p)=g(p),
\end{equation}
where $g$ is in $ C(K)$ and satisfies $\lim_{m\to \infty}\max_{p\in V_m\setminus V_0} |\Delta_m f(p)-g(p)|=0$.

The notations $C_0(K)$ and $\mathcal D_0$ denote respectively  the sets of  functions in $C(K)$ and $\D$ which vanish on $\partial K$ (See for instance \cite[Example 3.7.3]{Kigami} and \cite[Section 2]{FukushimaShima}). We will also consider the discrete measures on $\{V_m\}_{m\ge 0}$:
\begin{equation}\label{eq:measure}
\mu_m:=\frac2{3^{m+1}}\sum_{p\in V_m} \delta_p.
\end{equation}
For later use, we denote by $a_m$ the number $\frac{3^{m+1}}2$ and thus $\mu_m=\frac1{a_m}\sum_{p\in V_m} \delta_p$. 
%
%
%

For any function $f:V_*\to \mathbb R$, we consider the quadratic form
\[
\E_m(f,f)=\frac{5^m}{a_m}\sum_{x,y\in V_m, x\sim y} (f(x)-f(y))^2,
\]
where $x\sim y$ denotes that $x,y$ are neighbors in $G_m$.
Note that $\E_m(f,f)$ is non-decreasing in $m$. Define 
\[
\E(f,f)=\lim_{m\to +\infty} \E_m(f,f)
\]
and
\[
\F_0=\{f\in C(K):\lim_{m\to +\infty}\E_m(f,f)<\infty, f=0 \texton \partial K \}.
\]
By Theorem 4.1 and Lemma 4.1 in \cite{FukushimaShima}, $(\E,\F_0)$ is a local regular Dirichlet form on $L^2(K,\mu)$. Moreover, for any functions $f,g$ on $V_m$ vanishing on $\partial K$
\[
\E_m(f,g)=-\int_{V_m^0} \Delta_m f(x)g(x)d\mu_m(x),
\]
and for  $f \in \D_0, g\in \F_0$, 
\[
\E(f,g)=-\int_{K} \Delta f(x)g(x)d\mu(x).
\]
From \cite[Theorem 4.2]{FukushimaShima}  the Friedrichs extension of the Kigami Laplacian $\Delta$ is the self-adjoint operator on $L^2(K,\mu)$ which is the generator of $(\E,\F_0)$. We still denote this generator by $\Delta$ and the operator $\Delta$ with domain $\mathcal{D}(\Delta)$ is referred to as the Dirichlet Laplacian on $K$. 

%
%
%

The following lemma shows that any $u \in \mathcal{D}(\Delta)$ is H\"older continuous.

\begin{lem}\label{holder regu}
There exists a constant $C>0$ such that for every $u \in \mathcal{D}(\Delta)$ and $x,y \in K$, 
\[
| u(x)-u(y)| \le C d(x,y)^{d_w-d_h} \| \Delta u \|_{L^1(K,\mu)},
\]
where, as above, the parameter $d_h=\frac{\ln 3}{\ln 2}$ is the Hausdorff dimension and the parameter $d_w=\frac{\ln 5}{\ln 2}$  is the so-called walk dimension.
\end{lem}

\begin{proof}
Let $g^0$ be the reproducing kernel of the Dirichlet form $(\mathcal E, \mathcal F_0)$, see \cite[Theorem 4.1, (ii)]{FukushimaShima}. We have for every $y \in K$, and $u$ in $\mathcal F_0$,
\[
g^0( \cdot ,y) \in \mathcal F_0, \qquad \mathcal{E}( g^0( \cdot ,y), u) =u(y).
\]
For $u \in \mathcal D_0$, one obtains then
\begin{align*}
|u(x) -u(y)| & =\left|  \mathcal{E}( g^0( \cdot ,x), u) -\mathcal{E}( g^0( \cdot ,y), u) \right| \\
 &= \left| \int_K g^0( z ,x) \Delta u (z) d\mu(z) -\int_K g^0( z ,y) \Delta u (z) d\mu(z)   \right| \\
 &\le \int_K \left| g^0( z ,x) - g^0( z ,y)  \right| | \Delta u (z)| d\mu(z).
\end{align*}
Using \cite[Theorem 4.1, (GF4)]{Kigami2}, we have for every $x,y,z \in K$,
\[
 \left| g^0( z ,x) - g^0( z ,y)  \right| \le C d(x,y)^{d_w-d_h}.
\]
The result follows easily.
\end{proof}

\subsection{Discrete Fractional Gaussian Fields}

Let $0<\lambda_1^m\le \lambda_2^m\le \dots\le \lambda_{N_m}^m$ be the series of increasing eigenvalues (each being repeated according to its  multiplicity) of $-\Delta_m$ with zero boundary condition (see Section 3 in \cite{FukushimaShima}).   Let $(\Phi_i^m)_{1 \le i \le N_m}$ be the corresponding orthonormal eigenfunctions with respect to the measure $\mu_m$ defined in \eqref{eq:measure}.
%
%
The discrete Riesz kernel on $V_m$ with parameter $s \ge 0$ is defined by 
\begin{equation}\label{eq:DRiesz}
G_s^m(x,y)=\sum_{i=1}^{N_m} (\lambda_i^m)^{-s} \Phi_i^m(x)\Phi_i^m(y), \quad x,y \in V_m.
\end{equation}
From this definition it is clear that the matrix $( G_s^m(x,y))_{x,y \in V_m}$ is symmetric and non negative. It is therefore the covariance matrix of a Gaussian vector.

For $f\in \ell(V_m)$, the discrete fractional Laplacian $(-\Delta_m)^{-s}$ is defined by 
\begin{equation}\label{eq:DLaplacian}
(-\Delta_m)^{-s} f (x) =\frac1{a_m}\sum_{y\in V_m}G_s^m(x,y)f(y)
=\sum_{i=1}^{N_m} (\lambda_i^m)^{-s} \Phi_i^m(x) \frac1{a_m}\sum_{y\in V_m} \Phi_i^m(y) f(y).
\end{equation}
Note that 
$\|(-\Delta_m)^{-s}f\|_{ L^2(V_m,\mu_m)}\le (\lambda_1^m)^{-s}\|f\|_{ L^2(V_m,\mu_m)}$. 
Moreover,  one has $\inf_m \lambda_1^m>0$ from \cite[Lemma 5.2]{FukushimaShima}. Hence the operators $(-\Delta_m)^{-s}:L^2(V_m,\mu_m) \to L^2(V_m,\mu_m)$ are uniformly bounded.

\begin{defn}[DFGF]\label{def:DFGF}
Let $s \ge 0$. A discrete fractional Gaussian field $X_s^m$ with parameter $s$ on $V_m$ is a Gaussian vector indexed by $V_m$ with mean zero and covariance matrix $G_{2s}^m(x,y)$.
\end{defn}

\begin{defn}[Discrete log-correlated fields]
 We define the discrete log-correlated Gaussian field $X^m$ on $V_m$ as the discrete fractional Gaussian field $X_s^m$ with parameter $s=\frac{d_h}{2d_w}$; see Definition \ref{def:log} and Remark \ref{remklog} for further explanation about this terminology.
  \end{defn}

\begin{remark}
If $(W_i)_{1 \le i \le N_m}$ is a sequence of i.i.d Gaussian random variables with mean zero and variance one, then
\[
X_s^m(x):=\sum_{i=1}^{N_m} (\lambda_i^m)^{-s} \Phi_i^m(x)W_i , \quad x \in V_m
\]
is easily seen to be a DFGF with parameter $s$ on $V_m$.
\end{remark}

 For any $f\in \ell (V_m)$, we will use the notation
\[
X_s^m(f)=\frac1{a_m} \sum_{p\in V_m} f(p) X_s^m(p).
\]
We then note that for $f,g \in \ell (V_m)$
\begin{align}
\mathbb{E}( X_s^m(f)X_s^m(g))&=\frac1{a_m^2} \sum_{p,q \in V_m} f(p)g(q) \mathbb{E}( X_s^m(p)X_s^m(q))\nonumber\\
 & =\frac1{a_m^2} \sum_{p,q \in V_m} f(p)g(q) G_{2s}^m(p,q)\nonumber \\
 &= \frac1{a_m} \sum_{p \in V_m} f(p) \left( \frac1{a_m} \sum_{q \in V_m}g(q) G_{2s}^m(p,q)\right) \nonumber\\
 &= \frac1{a_m} \sum_{p \in V_m} f(p) (-\Delta_m )^{-2s} g (p) \label{eq:DCov}\\
 &= \frac1{a_m} \sum_{p \in V_m} (-\Delta_m )^{-s} f(p) (-\Delta_m )^{-s} g (p).\nonumber
 \end{align}

\subsection{Fractional Laplacians and fractional Riesz kernels}

The Laplacian $\Delta$ with domain $\mathcal{D}(\Delta)$ is the generator of a strongly continuous Markov semigroup $\{P_t\}_{t \ge 0}$ on $L^2(K,\mu)$. This semigroup admits a bicontinuous heat kernel $p_t(x,y)$, $t>0$, $x,y \in K$, with respect to the Hausdorff measure $\mu$. It is called the Dirichlet heat kernel on $K$. 

This heat kernel satisfies for some $c_{1},c_{2} \in(0,\infty)$,
\begin{equation}\label{eq:subGauss-upper}
p_{t}(x,y)\leq c_{1}t^{-\frac{d_{h}}{d_{w}}}\!\exp\biggl(\!-c_{2}\Bigl(\frac{d(x,y)^{d_{w}}}{t}\Bigr)^{\frac{1}{d_{w}-1}}\!\biggr)
\end{equation}
for every \ $(x,y)\in K \times K$ and $t\in\bigl(0,+\infty)$.  The exact values of $c_1,c_2$ are irrelevant in our analysis.   As above, the parameter $d_h=\frac{\ln 3}{\ln 2}$ is the Hausdorff dimension. The parameter $d_w=\frac{\ln 5}{\ln 2}$  is called the walk dimension.  The quantity $d_s=\frac{2d_h}{d_w}$ is often referred to as the spectral dimension. Since $d_w > 2$, one speaks of sub-Gaussian  heat kernel upper estimates.  

%
%

The  Dirichlet heat kernel $p_t(x,y)$ admits a uniformly convergent spectral expansion:
\begin{align}\label{spectral}
p_t(x,y)=\sum_{j=1}^{+\infty} e^{-\lambda_j t} \Phi_j(x) \Phi_j(y)
\end{align}
where $0<\lambda_1\le \lambda_2\le  \cdots \le \lambda_j \le \cdots$ are the eigenvalues of $-\Delta$ and  $(\Phi_j)_{j\ge 1} \subset  \mathcal{D} (\Delta)$ is an orthonormal basis of $L^2(K,\mu)$ such that  
\[
\Delta \Phi_j =-\lambda_j \Phi_j.
\]
Notice that $\Phi_j \in \mathcal{D}(\Delta)$ and thus is H\"older continuous. It is known from the work of Fukushima and Shima \cite{FukushimaShima} that the counting function of the eigenvalues:
\[
N(t)=\mathbf{Card} \{ \lambda_j \le t \}
\]
satisfies
\begin{equation}\label{Weyl_gasket}
N(t) \sim \Theta(t) t^{d_h/d_w}
\end{equation}
when $t \to +\infty$ where $\Theta$ is a function bounded away from 0. In particular,
\begin{align}\label{series_eigen}
\sum_{j=1}^{+\infty} \frac{1}{\lambda_j^{2s}} <+\infty
\end{align}
if and only if $s>\frac{d_h}{2d_w}$. We will consider the following space of test functions
\[
\mathcal{S}(K)= \left\{ f \in C_0(K),  \forall k \ge 0 \lim_{n \to +\infty} n^k \left| \int_K  \Phi_n(y) f(y) d\mu(y) \right| =0 \right\}.
\]

It is clear that if $f \in \mathcal{S}(K)$, then $f \in \bigcap_{k \ge 0}  \mathrm{dom}((-\Delta)^{k})$ and thus for every $k \ge 0$, $(-\Delta)^{k} f \in C(K)$ and is H\"older continuous. We also note that from \cite[Lemma 4.1(iii)]{FukushimaShima} $\mathcal{S}(K) \subset \mathcal D_0$. We consider then on $\mathcal{S}(K)$ the topology defined by the family of norms
\[
\| f \|_k = \| (-\Delta)^{k} f \|_{L^2(K,\mu)}, \quad k \ge 0.
\]
From \cite[Theorem 2]{Pietsch}, thanks to \eqref{series_eigen},  $\mathcal{S}(K)$ is a Fr\'echet nuclear space. The dual space of $\mathcal{S}(K)$  (for the latter topology) will be denoted $\mathcal{S}'(K)$.

\begin{defn}[Fractional Laplacians]\label{FLaplacian}
Let $s \ge 0$. For $f \in L^2(K,\mu)$, the fractional Laplacian $(-\Delta)^{-s}$ on $f$ is defined as
\[
(-\Delta)^{-s} f  =\sum_{j=1}^{+\infty} \frac{1}{\lambda_j^s} \Phi_j \int_K  \Phi_j(y) f(y) d\mu(y).
\]
For $f \in \mathcal{D}((-\Delta)^s):=\left\{ f \in L^2(K,\mu), \sum_{j=1}^{+\infty} \lambda_j^{2s} \left( \int_K  \Phi_j(y) f(y) d\mu(y)\right)^2 <\infty \right\}$,
the fractional Laplacian $(-\Delta)^{s}$ on $f$ is  defined as
\[
(-\Delta)^{s} f  =\sum_{j=1}^{+\infty} \lambda_j^s \Phi_j \int_K  \Phi_j(y) f(y) d\mu(y).
\]
\end{defn}

From the definition, it is clear that $(-\Delta)^{-s}: L^2(K,\mu) \to L^2(K,\mu)$ is a bounded operator. More precisely, one has $\|(-\Delta)^{-s}\|_{ L^2(K,\mu) \to L^2(K,\mu)}\le \lambda_1^{-s}$. 

\begin{defn}
	For a parameter $s \ge 0$, we define the fractional Riesz kernel $G_s $ by 
	\begin{align}\label{green function}
	G_s(x,y)=\frac{1}{\Gamma(s)} \int_0^{+\infty} t^{s-1} p_t(x,y) dt, \quad x,y \in K, \, x\neq y.
	\end{align}
	with $\Gamma$ the gamma function.
\end{defn}

\begin{remark}
We note that from  \eqref{eq:subGauss-upper} the integral \eqref{green function} is indeed convergent for all values of $s \ge 0$ provided that $ x \neq y$.
\end{remark}

We will be interested in the growth size of $G_s$. The following estimates are therefore important.

\begin{prop}\label{estimate G}
	\
	\begin{enumerate}
		\item If  $s \in [0, d_h/d_w)$, there exists a constant $C >0$ such that for every  $x,y \in K$, $x \neq y$,
		\[
		 G_s(x,y)  \le \frac{C}{d(x,y)^{d_h-sd_w}}.
		\]
		\item If $s = d_h/d_w$, there exists a constant $C >0$ such that for every  $x,y \in K$, $x \neq y$
		\[
		 G_s(x,y)  \le C | \ln d(x,y) |.
		\]
		\item If  $s > d_h/d_w$, there exists a constant $C >0$ such that for every  $x,y \in K$, 
		\[
		 G_s(x,y) \le C.
		\]
	\end{enumerate}

\end{prop}

\begin{proof}
The proof is similar to the proof of \cite[Proposition 2.6]{BaudoinLacaux}  (which dealt with the Neumann fractional Riesz kernels) and thus is omitted for conciseness.
\end{proof}

%
%
%
%
%
%
%

\begin{lem}
Let $s > 0$. For $f \in \mathcal{S}(K)$, and $x \in K$
\begin{equation}\label{eq:Fractional Laplacian-Riesz Kernel}
 (-\Delta)^{-s} f (x) = \int_K G_{s}(x,y) f(y) d\mu(y).
\end{equation}
\end{lem}

\begin{proof}
We first note that the integral  $\int_K G_{s}(x,y) f(y) d\mu(y)$ is indeed convergent. Since $\mathcal{S}(K) \subset C(K)$, it is enough to prove that for every $x \in K$, $\int_K G_{s}(x,y)  d\mu(y) <+\infty$ which easily follows from Proposition \ref{estimate G} because \eqref{Ahlfors} implies that for $\gamma < d_h$
\[
\int_K \frac{d\mu(y)}{d(x,y)^\gamma}<+\infty.
\]
Using then Fubini's theorem and the definition of $G_s$, one gets
\[
\int_K G_{s}(x,y) f(y) d\mu(y)=\frac{1}{\Gamma(s)} \int_0^{+\infty} t^{s-1} P_t f (x)  dt, 
\]
From \eqref{spectral}, it is seen that
\[
\frac{1}{\Gamma(s)} \int_0^{+\infty} t^{s-1} P_t f (x)  dt= (-\Delta)^{-s} f (x) 
\]
and the conclusion follows.
\end{proof}

\begin{remark}\label{rem:Riesz square}
For $s>\frac{d_h}{2d_w}$, we observe that \eqref{eq:Fractional Laplacian-Riesz Kernel} holds for any $f\in L^2(K,\mu)$ and  $\mu$ a.e. $x\in K$. Indeed, this can been seen from Proposition \ref{estimate G} since it implies that $G_s(x,\cdot)\in L^2(K,\mu)$ for every $x\in K$. We refer to \cite[Propositions 2.7 and 2.8]{BaudoinLacaux} for more details. Moreover, in this case we also have
\[
G_s(x,y)=\sum_{j=1}^{+\infty}\frac1{\lambda_j^s}\Phi_j(x)\Phi_j(y).
\]
\end{remark}

\begin{cor}
Let $s > 0$. For $f,g \in \mathcal{S}(K)$,
\[
 \int_{K} (-\Delta)^{-s} f (-\Delta)^{-s} g d\mu= \int_K \int_K f(x) g(y) G_{2s}(x,y) d\mu(x) d\mu(y).
\]

\end{cor}

\begin{proof}
From the definition of $(-\Delta)^{-s} $ and the previous lemma 
\begin{align*}
 \int_{K} (-\Delta)^{-s} f (-\Delta)^{-s} g d\mu &= \sum_{j=1}^{+\infty} \frac{1}{\lambda_j^{2s}}  \int_K  \Phi_j(y) f(y) d\mu(y)\int_K  \Phi_j(y) g(y) d\mu(y) \\
  &= \int_{K}  f (-\Delta)^{-2s} g d\mu \\
  &= \int_K \int_K f(x) g(y) G_{2s}(x,y) d\mu(x) d\mu(y).
\end{align*}
\end{proof}

\subsection{Fractional Gaussian Fields}

The next theorem states the existence of the fractional Gaussian fields.

\begin{thm}\label{existence FGF}
Let $s \ge 0$. There exists a centered Gaussian distribution $X_s$ on $\mathcal{S}'(K)$ such that for $f,g \in \mathcal{S}(K)$,
\[
\mathbb{E}( X_s(f) X_s(g))= \int_{K} (-\Delta)^{-s} f (-\Delta)^{-s} g d\mu.
\]

\end{thm}

\begin{proof}
The space $\mathcal{S}(K)$ is a nuclear space. By the Bochner-Minlos theorem in nuclear spaces, it is enough to prove that the functional
\[
\varphi: f \to \exp \left(-\frac{1}{2}\int_K | (-\Delta)^{-s} f |^2 d\mu\right)
\]
which is defined on $\mathcal{S}(K)$ is continuous  at 0 and positive definite.  Since the quadratic form $\int_K | (-\Delta)^{-s} f |^2 d\mu$ is positive definite on $\mathcal{S}(K)$, it follows from Proposition 2.4 in \cite{LSSW} that $\varphi$ is indeed definite positive. From Definition \ref{FLaplacian}, it is easy to see that there exists a constant $C>0$ such that for every $f \in \mathcal{S}(K)$,
 \[
 \int_K | (-\Delta)^{-s} f |^2 d\mu \le C \int_K f^2 d\mu.
 \]
 Since the convergence in $\mathcal{S}(K)$ implies the convergence in $L^2$, we conclude that $\varphi$ is indeed continuous at $0$.
%
\end{proof}

\begin{defn}[FGF]\label{def:FGF}
Let $s \ge 0$. A  fractional Gaussian field $X_s$ with parameter $s$ on $K$ is a centered Gaussian field $\left\{ X_s(f), f  \in \mathcal{S}(K)  \right\}$ such that 
for any $f,g \in \mathcal{S}(K)$,
\[
\mathbb{E}( X_s(f) X_s(g))= \int_{K} (-\Delta)^{-s} f (-\Delta)^{-s} g d\mu.
\]
\end{defn}


\begin{defn}[Log-correlated field]\label{def:log}
 We define a  log-correlated Gaussian field on $K$ as a fractional Gaussian field $X_s$ in Definition \ref{def:FGF} with the parameter $s=\frac{d_h}{2d_w}$.
  \end{defn}
  
\begin{remark}\label{remklog}
We use the terminology  log-correlated  field because of the estimate proved in Proposition \ref{estimate G} 
on the correlation function $G_{2s}$ for $s =\frac{d_h}{2d_w}$.
\end{remark}
  
%

In the following of this section, our aim is to establish that the FGF has an $L^2$ density if and only if $s>\frac{d_h}{2d_w}$. We begin with some reminders on Gaussian measures.

Let $\mathcal K$ be the Borel $\sigma$-field on $K$. Given a probability space $(\Omega, \mathcal F, \mathbb P)$, we consider a real-valued centered Gaussian random measure $W: \mathcal K \to L^2(\Omega, \mathcal F, \mathbb P)$ with density $\mu$ on $K$. Such measure is often referred to as white noise. In other words, $W$ is such that 
\begin{itemize}
    \item $W$ is a measure on $(K,\mathcal K)$ almost surely;
    \item For any $A\in \mathcal K$ of finite measure, $W(A)$ is a real-valued Gaussian variable with mean zero and variance $\mathbb E(W(A)^2)=\mu(A)$;
    \item For any sequence of pairwise disjoint measurable sets $(A_n)_{n\in \mathbb N}\in \mathcal \mathcal K^{\mathbb N}$, the random variables $W(A_n)$, $n\in \mathbb N$, are independent. 
\end{itemize}
Hence for any $f\in L^2(K,\mathcal K, \mu)$, the stochastic integral $W(f)=\int_K fdW$ is a well-defined centered Gaussian random variable. Moreover, the Gaussian measure $W$ gives rise to an isonormal Gaussian family $\{W(f),f\in L^2(K,\mu)\}$ with the covariance function 
\[
\mathbb E(W(f)W(g))=\int_K fgd\mu.
\]

Recall that the Riesz kernel $G_s(x,y)$ is square integrable for $s>\frac{d_h}{2d_w}$, see Remark \ref{rem:Riesz square}. We then introduce the following definition. 
\begin{defn}\label{def:density field}
Let $s>\frac{d_h}{2d_w}$. The fractional Brownian field with parameter $s$ is defined as 
\[
\tilde{X}_s(x)=\int_K G_s(x,y) W(dy), \quad x\in K.
\]
\end{defn}
From Remark \ref{rem:Riesz square}, one can equivalently define \[
\tilde{X}_s(x)=\sum_{i=1}^{+\infty} \lambda_i^{-s} \Phi_i(x)W_i,
\]
where $(W_i:=W(\Phi_i))_{i\ge 1}$ is an i.i.d. sequence of  Gaussian random variables with mean zero and variance one.

\begin{prop}\label{density field}
Let $s >\frac{d_h}{2d_w}$. Then the Gaussian random field defined by
	\[
	X_s(f)=\int_K f(x) \tilde{X}_s(x) d\mu (x), \qquad f \in \mathcal{S}(K),
	\]
has the law of a FGF with parameter $s$. Moreover, if there exists a Gaussian field $(\tilde{Y}_s(x))_{x\in K}$ on $K$ with 
\[
\mathbb{E}\left( \int_K \tilde{Y}_s(x)^2 d\mu(x)\right)<+\infty
\]
such that the Gaussian random field defined by
	\[
	Y_s(f)=\int_K f(x) \tilde{Y}_s(x) d\mu (x), \qquad f \in \mathcal{S}(K),
	\]
has the law of a FGF with parameter $s$, then  $s>\frac{d_h}{2d_w}$.
\end{prop}
\begin{proof}
Let us assume that $s>\frac{d_h}{2d_w}$. From Fubini's theorem, for every $f \in \mathcal{S}(K)$, one has a.s.
\[
\int_K f(x) \tilde{X}_s(x) d\mu (x)= \int_K f(x)\int_K G_s(x,y)W(dy)d\mu(x)=\int_K(-\Delta)^{-s}f(y)W(dy).
\]
Thus, $\int_K f(x) \tilde{X}_s(x) d\mu (x)$ is a Gaussian random variable with mean zero and variance
\[
\int_K |(-\Delta )^{-s} f (x)|^2 d\mu(x).
\]

On the other hand, assume that there exists a Gaussian field $(\tilde{Y}_s(x))_{x\in K}$ on $K$ with 
\[
\mathbb{E}\left( \int_K \tilde{Y}_s(x)^2 d\mu(x)\right)<+\infty
\] 
and such that the Gaussian random field defined by
	\[
	Y_s(f)=\int_K f(x) \tilde{Y}_s(x) d\mu (x), \qquad f \in \mathcal{S}(K),
	\]
has the law of a FGF with parameter $s$. Using spectral decomposition, we have
\[
\tilde{Y}_s(x)= \sum_{i=1}^{+\infty} Y_s(\Phi_i) \, \Phi_i (x).
\]

Notice that the sequence  $(Y_s(\Phi_i))_{i \ge 1}$ is a sequence of independent Gaussian random variables with mean zero and variance $(\lambda_i^{-2s})_{i\ge 1}$. Indeed, we recall that $(\Phi_i)_{i\ge 1}$ is an orthonormal basis in $L^2(K,\mu)$ and  $\Phi_i\in \mathcal S$. Then by Definition \ref{def:FGF}, 
\[
\mathbb E(Y_s(\Phi_i)Y_s(\Phi_j))=(\lambda_i\lambda_j)^{-s}\int_K \Phi_i \Phi_j d\mu=(\lambda_i\lambda_j)^{-s} \delta_{ij},
\]
where $\delta_{ij}$ is the Kronecker delta.
Since $\mathbb{E}\left( \int_K \tilde{Y}_s(x)^2 d\mu(x)\right)<+\infty$, we must have
\[
\sum_{j=1}^{+\infty} \frac{1}{\lambda_j^{2s}} <+\infty
\]
and therefore $s>\frac{d_h}{2d_w}$.
\end{proof}

\section{Regularity properties of the FGFs}

\subsection{Sobolev spaces}\label{def:Sobolev}

For any $ \alpha \ge 0$, we define the Sobolev  space $H^\alpha (K)$ as the closure of $\mathcal S (K)$ with respect to the norm 
\[
\|f\|^2_{H^\alpha(K)}:=\sum_{j=1}^{\infty} \lambda_j^\alpha \left(\int_K f\Phi_jd\mu\right)^2,
\]
and the corresponding inner product is 
\[
(f,g)_{H^\alpha(K)}=\sum_{j=1}^{\infty} \lambda_j^\alpha \int_K f\Phi_jd\mu \int_K g\Phi_jd\mu,
\]
where we recall that $(\lambda_j)_{j\ge 1}$ are the non-decreasing eigenvalues of the Laplacian $\Delta$ on $K$ and $(\Phi_j)_{j\ge 1}$ are the corresponding orthonormal eigenfunctions.
Denote by $H^{-\alpha}(K) \subset \mathcal{S}'(K)$ the  dual space of $H^\alpha (K)$ in the distributional sense. Then we have the following lemma.
\begin{lem}\label{lem:Sobolev dual}
The canonical norm on $H^{-\alpha}(K)$  induced by $H^{\alpha}(K)$ is given by 
\[
\|\psi\|^2_{H^{-\alpha}(K)}:=\sum_{j=1}^{\infty} \lambda_j^{-\alpha}  \psi (\Phi_j)^2, \quad \forall \psi \in H^{-\alpha}(K).
\]
\end{lem}
\begin{proof}
The proof is standard, we write it down for the sake of completeness. For every $\psi\in H^{-\alpha}(K)$ there exists $f_{\psi}\in H^\alpha(K)$ such that  $\psi (g)=(f_{\psi}, g)_{H^\alpha(K)}$ for all $g\in \mathcal S (K)$. In particular, the above  inner product gives that for every $\Phi_j, j\ge 1$, 
\[
\psi (\Phi_j)=(f_{\psi}, \Phi_j)_{H^\alpha(K)}=\lambda_j^{\alpha}\int_K f_{\psi} \Phi_jd\mu.
\]
Note also that by isometry one has $\|\psi\|_{H^{-\alpha}(K)}=\|f_{\psi}\|_{H^{\alpha}(K)}$. Hence
\[
\|\psi\|_{H^{-\alpha}(K)}^2=\|f_{\psi}\|_{H^{\alpha}(K)}^2=\sum_{j=1}^{\infty} \lambda_j^\alpha \left(\int_K f_{\psi}\Phi_j d\mu \right)^2
=\sum_{j=1}^{\infty} \lambda_j^{-\alpha}  \psi (\Phi_j)^2.
\]
\end{proof}

\subsection{Sobolev regularity property of the continuous FGFs in the range  \texorpdfstring{$0 \le s \le \frac{d_h}{2d_w}$}{s}}

\begin{prop}
Let $0 \le s \le \frac{d_h}{2d_w}$. Then the FGF $X_s$ a.s. belongs to $H^{-\alpha}(K)$ for every $\alpha > \frac{d_h}{d_w}-2s$. More precisely, for every $\alpha > \frac{d_h}{d_w}-2s$, the series
\[
\sum_{j=1}^{\infty} \lambda_j^{-\alpha}  X_s (\Phi_j)^2
\]
is a.s. convergent.
\end{prop}

\begin{proof}
The random variables $X_s (\Phi_j)$ are independent Gaussian random variables with mean zero and variance $\lambda_j^{-2s}$. Since the series
\[
\sum_{j=1}^{\infty} \lambda_j^{-\alpha-2s}
\]
converges for $\alpha +2s >\frac{d_h}{d_w}$, the result follows.
\end{proof}

\subsection{H\"older regularity property of the continuous FGFs in the range  \texorpdfstring{$s>\frac{d_h}{2d_w}$}{s}}

%
%
%

In this section our goal is to study the regularity of the density field $(\tilde{X}_s(x))_{x\in K}$ that appeared in  Definition \ref{def:density field}.  The following first result is almost immediate.

\begin{prop}
Let $s > \frac{d_h}{2d_w}$. The Gaussian field $(\tilde{X}_s(x))_{x\in K}$ on $K$ defined in Definition \ref{def:density field} is such that a.s.  $\tilde{X}_s \in H^{\alpha}(K)$ for every $\alpha < 2s -\frac{d_h}{d_w}$.
\end{prop}

\begin{proof}
As remarked before, one has
\[
\tilde{X}_s(x)=\sum_{i=1}^{+\infty} \lambda_i^{-s} \Phi_i(x)W_i,
\]
where $W_i$ is an i.i.d. sequence of  Gaussian random variables with mean zero and variance one. Therefore, 
\[
\|\tilde{X}_s \|^2_{H^\alpha(K)}=\sum_{j=1}^{\infty} \lambda_j^{\alpha -2s} W_j^2,
\]
which is a.s. finite if $\alpha < 2s -\frac{d_h}{d_w}$.
\end{proof}

Next, we are interested in the H\"older regularity properties  of $(\tilde{X}_s(x))_{x\in K}$. This requires a deeper analysis and our main analytical ingredients are the following H\"older regularization estimates for the operators $(-\Delta)^{-s}$.

\begin{thm}\label{regularity fractional Laplacian}

\

\begin{itemize}
\item Let $\frac{d_h}{2d_w} < s < 1-\frac{d_h}{2d_w} $. There exists a constant $C>0$ such that for every $f \in L^2(K,\mu)$ and $x,y \in K$,
\[
| (-\Delta)^{-s}f (x) -(-\Delta)^{-s}f (y)| \le C d(x,y)^{sd_w-\frac{d_h}{2}} \| f \|_{L^2(K,\mu)}.
\]
\item Let $ s = 1-\frac{d_h}{2d_w} $. There exists a constant $C>0$ such that for every $f \in L^2(K,\mu)$ and $x,y \in K$ with $d(x,y) \le 1/2$,
\[
| (-\Delta)^{-s}f (x) -(-\Delta)^{-s}f (y)| \le C d(x,y)^{d_w-d_h} | \ln d(x,y)|  \| f \|_{L^2(K,\mu)}.
\]
\item Let $ s > 1-\frac{d_h}{2d_w} $. There exists a constant $C>0$ such that for every $f \in L^2(K,\mu)$ and $x,y \in K$,
\[
| (-\Delta)^{-s}f (x) -(-\Delta)^{-s}f (y)| \le  C d(x,y)^{d_w-d_h}  \| f \|_{L^2(K,\mu)}.
\]
\end{itemize}

\end{thm}

The proof of the Theorem is based on the following lemmas. In the first lemma below, we use a slight modification of an argument due to Barlow in \cite[Theorem 3.40]{Barlow}, but we include the proof for the sake of completeness. For $\lambda  \ge 0$, let $U_\lambda=(\lambda-\Delta)^{-1}$ be the resolvent operator and let $u_\lambda (x,y)$, $x,y \in K$ be its kernel.

\begin{lem}
There exists a constant $C>0$ such that for every $x,y,z\in K$ and $\lambda >0$
 \[
 | u_\lambda (x,z)-u_\lambda (y,z) | \le C d(x,y)^{d_w-d_h}
 \]
 and
 \[
 \int_K | u_\lambda (x,z)-u_\lambda (y,z) | d\mu(z) \le C \lambda^{-\frac{d_h}{d_w}} d(x,y)^{d_w-d_h}.
 \]
\end{lem}

\begin{proof}
We slightly adapt the proof of Theorem 3.40 in \cite{Barlow}.  Let $\left( (X_t)_{t \ge 0}, (\mathbb P_x)_{x \in K}\right)$ be the Brownian motion on $K$, see \cite{Barlow}.  Note that the Dirichlet Laplacian $\Delta$ is  (twice) the generator of the process $(X_t)_{t \ge 0}$ killed at the boundary $V_0$.  From (3.39) in \cite{Barlow} we have
 \[
 u_\lambda(x,y)=q_\lambda (x,y) u_\lambda (y,y),
 \]
 with
 \[
 q_\lambda (x,y)=\mathbb{P}_x (T_y \le T_{V_0} \wedge R_\lambda )
 \]
 where $T_y$ is the hitting time of $y$, $T_{V_0}$ is the hitting time of $V_0$ and $R_\lambda$ is an independent exponential random variable with parameter $\lambda >0$. As in the proof  of Theorem 3.40 in \cite{Barlow} we have then
 \[
 | u_\lambda (x,z)-u_\lambda (y,z) |\le C ( q_\lambda (z,x) +q_\lambda (z,y))d(x,y)^{d_w-d_h}.
 \]
 The first estimate 
 \[
 | u_\lambda (x,z)-u_\lambda (y,z) | \le C d(x,y)^{d_w-d_h}
 \]
 follows easily since $q_\lambda (z,x)\le 1$.
 Using the on-diagonal lower bound for the transition densities of the Brownian motion $\left( (X_t)_{t \ge 0}, (\mathbb P_x)_{x \in K}\right)$, see \cite{Barlow,BarlowPerkins}, we easily obtain 
 \[
\int_K q_\lambda (z,x)d\mu(z) \le C \lambda^{-\frac{d_h}{d_w}},
 \]
 which yields the second estimate.
\end{proof}

\begin{lem}
There exists a constant $C>0$ such that for every $f \in L^2 (K,\mu)$, and $t>0$,
\[
| P_t f (x)-P_tf(y) | \le C e^{-\frac{\lambda_1}{2} t} \frac{d(x,y)^{d_w-d_h}}{t^{1-\frac{d_h}{2d_w}}} \| f \|_{L^2 (K,\mu)}.
\]
\end{lem}
\begin{proof}
We have from the previous lemma
 \[
 | u_\lambda (x,z)-u_\lambda (y,z) | \le C d(x,y)^{d_w-d_h}
 \]
 and
  \[
 \int_K | u_\lambda (x,z)-u_\lambda (y,z) | d\mu(z) \le C \lambda^{-\frac{d_h}{d_w}} d(x,y)^{d_w-d_h}.
 \]
 This yields
 \[
 \int_K | u_\lambda (x,z)-u_\lambda (y,z) |^2 d\mu(z) \le C \lambda^{-\frac{d_h}{d_w}} d(x,y)^{2(d_w-d_h)},
 \]
 from which we deduce
 \[
 | U_\lambda f (x) -U_\lambda f (y)| \le C \lambda^{-\frac{d_h}{2d_w}} d(x,y)^{d_w-d_h} \| f \|_{L^2 (K,\mu)}.
 \]
 This gives that for every $t>0$ and $\lambda >0$
 \begin{align}\label{ref_iu}
 | P_t f (x) -P_t f (y)|\le C \lambda^{-\frac{d_h}{2d_w}} d(x,y)^{d_w-d_h} \| (\Delta - \lambda)P_t f \|_{L^2 (K,\mu)}.
 \end{align}
 From spectral theory, one has
 \[
 \| (\Delta - \lambda)P_t f \|_{L^2 (K,\mu)} \le C\left( \frac{1}{t} + \lambda \right) e^{-\frac{\lambda_1}2 t} \| f \|_{L^2 (K,\mu)}.
 \]
 From \eqref{ref_iu}, we deduce then
 \[
| P_t f (x)-P_tf(y) | \le C e^{-\frac{\lambda_1}2 t} \frac{d(x,y)^{d_w-d_h}}{t^{1-\frac{d_h}{2d_w}}} \| f \|_{L^2 (K,\mu)}.
\]
by choosing $\lambda=\frac{1}{t}$.
\end{proof}

\begin{proof}[Proof of Theorem \ref{regularity fractional Laplacian}]

We first consider the case $\frac{d_h}{2d_w} < s < 1-\frac{d_h}{2d_w} $. Note that
\[
| (-\Delta)^{-s}f (x) -(-\Delta)^{-s}f (y)| \le \frac{1}{\Gamma(s)} \int_0^{+\infty} t^{s-1} | P_t f(x)-P_t f(y)| dt.
\]
We then split the integral into two parts:
\[
\int_0^{+\infty} t^{s-1} | P_t f(x)-P_t f(y)| dt=\int_0^{\delta} t^{s-1} | P_t f(x)-P_t f(y)| dt +\int_\delta^{+\infty} t^{s-1} | P_t f(x)-P_t f(y)| dt,
\]
where $\delta>0$ is a parameter to be chosen later. We first have
\begin{align*}
\int_0^{\delta} t^{s-1} | P_t f(x)-P_t f(y)| dt&\le \int_0^{\delta} t^{s-1}( | P_t f(x)|+|P_t f(y)|) dt \\
 & \le \int_0^{\delta} t^{s-1}  \frac{C}{t^{\frac{d_h}{2d_w}}}  dt \, \| f \|_{L^2 (K,\mu)} \\
  & \le C \delta^{s -\frac{d_h}{2d_w}} \| f \|_{L^2 (K,\mu)}.
\end{align*}
For the second integral, we have
\begin{align*}
\int_\delta^{+\infty} t^{s-1} | P_t f(x)-P_t f(y)| dt &\le C  \int_\delta^{+\infty} t^{s-1} e^{-\frac{\lambda_1}{2} t} \frac{d(x,y)^{d_w-d_h}}{t^{1-\frac{d_h}{2d_w}}}  dt \, \| f \|_{L^2 (K,\mu)} \\
 & \le C \delta^{s -1+\frac{d_h}{2d_w}} d(x,y)^{d_w-d_h} \| f \|_{L^2 (K,\mu)}.
\end{align*}
Choosing $\delta=d(x,y)^{d_w}$ finishes the proof that
\[
| (-\Delta)^{-s}f (x) -(-\Delta)^{-s}f (y)| \le C d(x,y)^{sd_w-\frac{d_h}{2}} \| f \|_{L^2(K,\mu)}.
\]

Next consider the case $s=1-\frac{d_h}{2d_w} $. The above method can still be used, but we now estimate the second integral as follows for $\delta \le \frac{1}{2^{d_w}}$
\[
\int_{\delta}^{+\infty} t^{-1} e^{-\frac{\lambda_1}{2} t} dt \le C | \ln \delta |.
\]
This yields
\[
| (-\Delta)^{-s}f (x) -(-\Delta)^{-s}f (y)| \le C d(x,y)^{d_w-d_h} | \ln d(x,y)|  \| f \|_{L^2(K,\mu)}.
\]

Finally, for the case $s>1-\frac{d_h}{2d_w}$, we just argue as follows
\begin{align*}
| (-\Delta)^{-s}f (x) -(-\Delta)^{-s}f (y)| & \le \frac{1}{\Gamma(s)} \int_0^{+\infty} t^{s-1} | P_t f(x)-P_t f(y)| dt \\
 & \le C  \int_0^{+\infty} t^{s-1} e^{-\frac{\lambda_1}{2} t} \frac{d(x,y)^{d_w-d_h}}{t^{1-\frac{d_h}{2d_w}}} dt \,  \| f \|_{L^2 (K,\mu)} \\
 &  \le C d(x,y)^{d_w-d_h}  \| f \|_{L^2(K,\mu)}.
\end{align*}
\end{proof}

As a consequence of Theorem \ref{regularity fractional Laplacian}, we obtain the H\"older regularization properties of the Riesz kernels.
\begin{cor}\label{regularity Riesz kernel}
\

\begin{itemize}
\item Let $\frac{d_h}{2d_w} < s < 1-\frac{d_h}{2d_w} $. There exists a constant $C>0$ such that for every $x,y \in K$,
\[
\int_K(G_s(x,z)-G_s(y,z) )^2d\mu(z) \le C d(x,y)^{2sd_w-d_h}.
\]
\item Let $ s = 1-\frac{d_h}{2d_w} $. There exists a constant $C>0$ such that for every $f \in L^2(K,\mu)$ and $x,y \in K$ with $d(x,y) \le 1/2$,
\[
\int_K(G_s(x,z)-G_s(y,z) )^2d\mu(z)  \le C d(x,y)^{2(d_w-d_h)} | \ln d(x,y)|^2.
\]
\item Let $ s > 1-\frac{d_h}{2d_w} $. There exists a constant $C>0$ such that for every $x,y \in K$,
\[
\int_K(G_s(x,z)-G_s(y,z) )^2d\mu(z)  \le  C d(x,y)^{2(d_w-d_h)}.
\]
\end{itemize}
\end{cor}
\begin{proof}
Recall that for any $f\in L^2(K,\mu)$, 
\[
\int_K (G_s(x,z)-G_s(y,z) ) f(z) d\mu(z)=(-\Delta)^{-s}f(x)-(-\Delta)^{-s}f(y).
\]
By $L^2$ duality, we conclude the results from Theorem \ref{regularity fractional Laplacian}.
\end{proof}

We are now ready for the main results of this section:

\begin{thm}
Let $s > \frac{d_h}{2d_w}$ and denote $H_s=\min (sd_w -d_h/2, d_w-d_h)$ and 
\begin{align*}
\omega_s(x,y) =
\begin{cases}
 d(x,y)^{H_s} \sqrt{\left|\ln d(x,y)\right|} , \, s \neq 1-\frac{d_h}{2d_w} \\
 d(x,y)^{d_w-d_h} \left|\ln d(x,y)\right|^{3/2}, \, s = 1-\frac{d_h}{2d_w}.
\end{cases}
\end{align*} 
There exists a continuous Gaussian field $(\tilde{X}_s(x))_{x \in K}$ such that  a.s.
	$$
	\lim_{\delta \to 0}\underset{\underset{x,y\in K}{\scriptsize 0< d(x,y)}\le \delta}{\sup}\  \frac{\left| \tilde{X}_s(x)-\tilde{X}_s(y)\right|}{\omega_s(x,y)} <+\infty	$$ 
	and such that the Gaussian random  field defined by
	\[
	X_s(f)=\int_K f(x) \tilde{X}_s(x) d\mu (x), \qquad f \in \mathcal{S}(K),
	\]
has the law of a FGF with parameter $s$.	
\end{thm}

\begin{proof}
Let $\tilde{X}_s(x)=\int_KG_s(x,y)W(dy)$, see Definition \ref{def:density field}. It follows from Proposition \ref{density field} that $\{X_s(f), f\in \mathcal S(K)\}$ has the law of FGF with parameter $s$.

Next we will use the entropy method as in \cite{AdlerTaylor}, see also \cite[Theorem 3.8]{BaudoinLacaux} to construct an appropriate continuous modification of $\tilde{X}_s$ that we will still denote by $\tilde{X}_s$ . Assume first  $s\ne 1-\frac{d_h}{2d_w}$. We observe that Corollary \ref{regularity Riesz kernel} gives
\[
\mathbb E \big((\tilde{X}_s(x)-\tilde{X}_s(y))^2\big) \le C d(x,y)^{2H_s}, \quad \forall x,y \in K.
\]
Consider the pseudo distance $\rho_s(x,y)$ defined by 
\[
\rho_s(x,y)=\sqrt{\mathbb E \big((\tilde{X}_s(x)-\tilde{X}_s(y))^2\big)}.
\] 
Then $\rho_s(x,y)\le Cd(x,y)^{H_s}$. Denote by $\mathcal N_{\rho_s}(\varepsilon)$ the smallest number of $\rho_s$-balls with radius $r\le \varepsilon$ that cover $K$. We set the log-entropy for $K$ by 
\[
\mathcal H_{\rho_s}(\varepsilon)=\ln (\mathcal N_{\rho_s}(\varepsilon)).
\]
According to \cite[Theorem 1.3.5]{AdlerTaylor}, there exist a random variable $\eta$ and a universal constant $D$ such that for all $\tau<\eta$,
\[
\sup_{\substack{\rho_s(x,y)\le \tau \\ x,y \in K}} |\tilde{X}_s(x)-\tilde{X}_s(y)| 
\le D \int_0^{\tau} \sqrt{\mathcal  H_{\rho_s}(\varepsilon)} d\varepsilon.
\]
Notice that $\mathcal N_{\rho_s}(\varepsilon)=O(\varepsilon^{-d_h/H_s})$. Then up to the change of $\eta$ and $D$, one has for all $\delta<\eta$,
\[
\sup_{\substack{d(x,y)\le \delta \\ x,y \in K}} |\tilde{X}_s(x)-\tilde{X}_s(y)| 
\le D \int_0^{C\delta^{H_s}} \sqrt{-\ln \varepsilon}\, d\varepsilon.
\]
Finally,  up the the change of constant $D$, for all $\delta<\eta$ small enough, we obtain from integration by parts that 
\begin{align*}
\sup_{\substack{d(x,y)\le \delta \\ x,y \in K}} |\tilde{X}_s(x)-\tilde{X}_s(y)| 
&\le D \Big(\delta^{H_s} \sqrt{-\ln \delta}+\int_0^{C\delta^{H_s}} \frac1{\sqrt{-\ln \varepsilon}}\, d\varepsilon\Big)
\le 2D \delta^{H_s} \sqrt{-\ln \delta}.
\end{align*}
Thus the proof is concluded.

Consider now the critical case  $s= 1-\frac{d_h}{2d_w}$. By Corollary \ref{regularity Riesz kernel}, we have  
\[
\rho_s(x,y)\le d(x,y)^{d_w-d_h}|\ln d(x,y)|=:F(d(x,y)).
\]
Observe that $F(t)$ is increasing on the interval $(0,t_0)$ for some small $t_0$. We denote by $F^{-1}$ the inverse function on the domain $(0, F(t_0))$.
Then for any $0<\varepsilon<F(t_0)$, one has $\mathcal N_{\rho_s}(\varepsilon)=O((F^{-1}(\varepsilon))^{-d_h})$. Using the same argument as above,  there exist a random variable $\eta$ and constants $C,D>0$ such that  for all $\delta<\min\{\eta, t_0\}$,
\[
\sup_{\substack{d(x,y)\le \delta \\ x,y \in K}} |\tilde{X}_s(x)-\tilde{X}_s(y)| 
\le D \int_0^{CF(\delta)} \sqrt{-\ln F^{-1}(\varepsilon)}\, d\varepsilon.
\]
Hence,  up to the change of constant $D$ and for all $\delta<\min\{\eta, t_0\}$ small enough, we have
\begin{align*}
\sup_{\substack{d(x,y)\le \delta \\ x,y \in K}} |\tilde{X}_s(x)-\tilde{X}_s(y)| 
&\le D \Big(F(\delta) \sqrt{-\ln \delta}-\int_0^{CF(\delta)} \varepsilon\big(\sqrt{-\ln F^{-1}(\varepsilon)}\big)'\, d\varepsilon\Big)
\le D F(\delta) \sqrt{-\ln \delta}.
\end{align*}
The second inequality follows from elementary computations below where we take $\gamma=d_w-d_h$ and let $\varepsilon$ be small enough:
\begin{align*}
\varepsilon\big(\sqrt{-\ln F^{-1}(\varepsilon)}\big)'
&=\frac{1}{2\sqrt{-\ln F^{-1}(\varepsilon)}}\frac{-\varepsilon}{F^{-1}(\varepsilon) F'(F^{-1}(\varepsilon))}
\\ &=
\frac{1}{2\sqrt{-\ln F^{-1}(\varepsilon)}}\frac{-\varepsilon}{(F^{-1}(\varepsilon))^{\gamma}(-\gamma \ln F^{-1}(\varepsilon)-1)}
\\ &=
\frac{1}{2\sqrt{-\ln F^{-1}(\varepsilon)}}\frac{-\varepsilon}{(\gamma F( F^{-1}(\varepsilon))-(F^{-1}(\varepsilon))^{\gamma})}
=O\left( \frac{-1}{\sqrt{-\ln F^{-1}(\varepsilon)}}\right).
\end{align*}
Thus we conclude that for $s=1-\frac{d_h}{2d_w}$
\[
\lim_{\delta \to 0}\underset{\underset{x,y\in K}{\scriptsize 0< d(x,y)}\le \delta}{\sup}\  \frac{\left| \tilde{X}_s(x)-\tilde{X}_s(y)\right|}{d(x,y)^{d_w-d_h} (\left|\ln d(x,y)\right|)^{3/2}} <+\infty.
\]
\end{proof}

For $s >1 $ the above result can substantially be improved.

\begin{prop}
Let $s >1 $. There exists a continuous Gaussian field $(\tilde{X}_s(x))_{x \in K}$ such that 
	$$
\mathbb{E} \left( \left(\sup_{x,y \in K, x\neq y} 	  \frac{\left| \tilde{X}_s(x)-\tilde{X}_s(y)\right|}{d(x,y)^{d_w-d_h} } \right)^2 \right) <+\infty	$$ 
	and such that the Gaussian random random field defined by
	\[
	X_s(f)=\int_K f(x) \tilde{X}_s(x) d\mu (x), \qquad f \in \mathcal{S}(K),
	\]
has the law of a FGF with parameter $s$.
\end{prop}

\begin{proof}
Let $s >1$. As above, let  $(\tilde{X}_s(x))_{x \in K}$ be  a continuous Gaussian field on $K$ such that the Gaussian random random field defined by
	\[
	X_s(f)=\int_K f(x) \tilde{X}_s(x) d\mu (x), \qquad f \in \mathcal{S}(K)
	\]
has the law of a FGF with parameter $s$. 
We have then
 \[
\tilde{X}_s(x)=\sum_{i=1}^{+\infty} \lambda_i^{-s} \Phi_i(x)W_i,
\]
where the $W_i$'s form an i.i.d. sequence of Gaussian random variables with mean zero and variance one. Let $\alpha >1-\frac{d_h}{2d_w}$ such that $s-\alpha> \frac{d_h}{2d_w}$. Since
\[
\mathbb{E}\left( \| (-\Delta)^{\alpha} \tilde{X}_s  \|^2_{L^2(K,\mu)}\right)=\sum_{i=1}^{+\infty} \lambda_i^{2(\alpha-s)} <+\infty,
\]
we deduce that $\tilde{X}_s$ almost surely belongs to the $L^2$ domain of $(-\Delta)^{\alpha} $, i.e.
\[
\mathbb{P} \left( \| (-\Delta)^{\alpha} \tilde{X}_s  \|^2_{L^2(K,\mu)} <+\infty \right)=1.
\]

From Theorem \ref{regularity fractional Laplacian}, one deduces
\[
\left| \tilde{X}_s (x) - \tilde{X}_s (y) \right| \le C_s d(x,y)^{d_w-d_h}\| (-\Delta)^{\alpha} \tilde{X}_s  \|_{L^2(K,\mu)},
\]
and the result follows.
\end{proof}

\section{Convergence of the discrete fields to the continuous fields}

In this section, our first main goal is to show for $s \ge 0$ the convergence in distribution in $\mathcal S'(K)$ of the approximations of discrete fractional Gaussian fields  on $V_m$ to the fractional Gaussian field  on the Sierpinski gasket. Our second goal will be to prove convergence in the Sobolev spaces $H^\alpha (K)$.

\subsection{Preliminary lemmas}

This section collects several lemmas that will later be needed.

\begin{lem}[\protect{\cite[Lemma 1.1]{BarlowPerkins}}]\label{lem:measure}
The sequence of measures $\{\mu_m\}_{m\ge 0}$ defined in \eqref{eq:measure} converges to the normalized Hausdorff measure $\mu$ on the Sierpinski gasket $K$ in the weak topology. That is, 
 \[
 \lim_{m\to \infty} \int_K fd\mu_m= \int_K fd\mu, \quad \forall f\in C(K).
 \]
\end{lem}
%
\begin{remark} \label{rem:discrete integral}
 Without abuse of notation, for any $g\in \ell(V_m)$, we may  write 
 \[
 \frac1{a_m}\sum_{p\in V_m} g(p)=\int_{V_m} g d\mu_m.
 \]
 Hence let $f_m=f|_{V_m}$ for $f\in C(K)$ in the lemma, one also has $ \lim_{m\to \infty} \int_{V_m} f_md\mu_m= \int_K fd\mu$. 
\end{remark}
\begin{lem}
[Convergence of discrete semigroups]\label{lem:HKconvergence}
 For all $f\in \mathcal S(K)$ and $t > 0$,
\[
\lim_{m\to \infty} \frac1{a_m}\sum_{p\in V_m} f_m(p)  P_t^m f_m(p)=\int_{K}f(x) P_tf(x) d\mu(x),
\]
where $f_m=f|_{V_m}$.
\end{lem}

\begin{proof}
We follow the strategy in \cite[Section 3.2.2]{CiprianiGinkel}. Recall the definition of Laplacian on $K$ in \eqref{eq:Laplacian} (see also  \cite[page 6]{FukushimaShima}), then for any $f\in \mathcal D_0$.
\[
\lim_{m\to \infty} \sup_{p\in V_m} |\Delta_mf_{m}(p)-\Delta f(p)|=0.
\]
 It follows from \cite[Theorem 2.1]{Kurtz} that for every $t \ge 0$
 \begin{equation}\label{eq:semigroup}
 \lim_{m\to \infty} \sup_{p\in V_m} |P_t^mf_m(p)-P_t f(p)|=0.
 \end{equation}
Indeed, the Laplacian on $K$ coincides with the extended limit of the sequence of operators $\{\Delta_m\}_{m\ge 0}$ defined in \cite[page 355]{Kurtz}.
Write 
\[
\frac1{a_m}\sum_{p\in V_m} f_m(p)P_t^m f_m(p)=\int_{V_m}f_mP_t^m f_md\mu_m,
\] 
and further
\[
\int_{V_m}f_mP_t^m f_md\mu_m=\int_{V_m}f_m\left(P_t^m f_m-(P_tf)_m\right)d\mu_m+\int_{V_m} f_m(P_tf)_md\mu_m.
\]
Taking the limit $m\to \infty$, the first term goes to zero from \eqref{eq:semigroup}. On the other hand, Lemma \ref{lem:measure} gives that $\int_{V_m}f_m(P_tf)_md\mu_m\to \int_K f P_t f d\mu$ and the proof is complete. 
\end{proof}
%

\begin{lem}
\label{lem:IPconvergence}
For all $f\in \mathcal S(K)$ and $s \ge 0$, when $m \to \infty$
\begin{equation}\label{eq:convergence}
\frac1{a_m} \sum_{p\in V_m} f_m(p) (-\Delta_m)^{-2s} f_m(p) \longrightarrow  \int_K f(x)(-\Delta)^{-2s}f(x)d\mu(x),
\end{equation}
where $f_m=f|_{V_m}$.
\end{lem}

\begin{proof}
For $s=0$, the result follows immediately from Lemma \ref{lem:measure}. We now assume $s>0$. Notice that 
\[
(-\Delta_m)^{-2s}f_m =\frac{1}{\Gamma(2s)} \int_0^{+\infty} t^{2s-1} P_t^m f_m dt.
\]
By Fubini's theorem, we therefore have
\[
\frac1{a_m} \sum_{p\in V_m} f_m(p)(-\Delta_m)^{-2s}f_m(p) =\frac{1}{\Gamma(2s)} \int_0^{+\infty} t^{2s-1} \frac1{a_m} \sum_{p\in V_m} f_m(p)P_t^m f_m(p) dt.
\]
Let us now note that by spectral theory
\[
\frac1{a_m} \sum_{p\in V_m} f_m(p)P_t^m f_m(p) \le e^{-\lambda_1^m t} \frac1{a_m} \sum_{p\in V_m} f_m(p)^2.
\]
Notice that 
$\sup_m  \frac1{a_m} \sum_{p\in V_m} f_m(p)^2 <+\infty$. Hence we deduce from  the dominated convergence theorem and Lemma \ref{lem:HKconvergence}  that 
\begin{align*}
\lim_{m\to \infty}\frac1{a_m} \sum_{p\in V_m} f_m(p) (-\Delta_m)^{-2s} f_m(p)
&=
\frac{1}{\Gamma(2s)} \int_0^{+\infty} t^{2s-1} \lim_{m\to \infty}\frac1{a_m} \sum_{p\in V_m} f_m(p)P_t^m f_m(p) dt
\\ &=
\frac{1}{\Gamma(2s)} \int_0^{+\infty} t^{2s-1}  \int_K f(x)P_tf(x)d\mu(x) dt
\\ &=
\int_K f(x)(-\Delta)^{-2s}f(x)d\mu(x).
\end{align*}

\end{proof}

\subsection{Convergence in distribution in  \texorpdfstring{$\mathcal S'(K)$}{SK}}

We are now in position to prove the following result.

\begin{thm}
Let $s \ge 0$. When $m \to \infty$, $X_s^m$ converges to $X_s$ in distribution in $\mathcal S'(K)$.
\end{thm}

\begin{proof}
We aim to prove $X_s^m \to X_s$ in law in $\mathcal S'(K)$. Since $\mathcal S(K)$ is a nuclear space, it suffices to prove the convergence of the characteristic functional (see for instance \cite[Th\'eor\`eme 2]{Meyer}). That is, for every $f\in \mathcal S(K)$, when $m \to \infty$
\[
\mathbb E \left[ \exp\left(i X_s^m(f_m)\right)\right]\longrightarrow \mathbb E \left[ \exp\left(i X_s(f)\right)\right],
\]
 where $f_m=f|_{V_m}$. It follows from \eqref{eq:DCov} that
\begin{align*}
\mathbb E  \left[\exp\left(i \,X_s^m(f_m)\right)  \right]
=\exp\left(-\frac12 \mathbb E \left((X_s^m(f_m))^2\right)\right) 
=\exp\left(-\frac1{2a_m} \sum_{p\in V_m} f_m(p)(-\Delta_m)^{-2s} f_m(p)\right).
\end{align*}
Similarly, Definition \ref{FLaplacian} gives that $  \mathbb E \left[  \exp\left(i X_s(f)\right) \right]=\exp\left(-\frac12 \int_K f(-\Delta)^{-2s} fd\mu\right)$. The conclusion therefore follows from Lemma \ref{lem:IPconvergence}.
\end{proof}

\subsection{Convergence in distribution in  Sobolev spaces}
Recall the Sobolev space $H^{\alpha}$ and the dual space $H^{-\alpha}$ defined in Section 3.1. In this section, we aim to prove the convergence of lifted DFGF  in the Sobolev space $H^{-\alpha}$ for appropriate $\alpha>0$. Following the scheme in \cite{CiprianiGinkel}, we first lift $X_s^m$ on $V_m$ to $K$ using Voronoi cells
defined by 
\[
C_p^m=\{x\in K:d(x,p)\le d(x,q), \forall q\in V_m\}, \quad p\in V_m.
\]
Equivalently, one has $C_p^m=\{x\in K:d(x,p)\le  2^{-(m+1)}\}$. 
\begin{defn}[DFGF in $H^{-\alpha}(K)$]
Let $X_s^m$ be the DFGF on  $V_m$ as in Definition \ref{def:DFGF}. We define $\bar X_s^m\in H^{-\alpha}(K)$ such that for $f\in H^{\alpha}(K)$
\[
\bar X_s^{\,m}(f)=\frac1{a_m}\sum_{p\in V_m} X_s^m(p)\bar f_m(p)=X_s^m(\bar f_m),
\]
where $\bar f_m(p):=\frac{1}{\mu(C_p^m)}\int_{C_p^m} f(x)d\mu(x)$ for any $p\in V_m$.
\end{defn}

Our main result in this section is the following theorem.
\begin{thm}\label{thm: CovergenceLaw}
Let $s \ge 0$. The Gaussian fields $\bar X_s^{\,m}$ converge in law to $X_s$  in the strong topology of $H^{-\alpha}(K)$ for $\alpha>2d_h/d_w$.
\end{thm}

Throughout the section we assume that $s \ge 0$. The proof is divided into two parts. We will first show  the tightness of the sequence $(\bar X_s^{\,m})_{m\ge 1}$ in $H^{-\alpha}(K)$. Thus every sequence has a convergent subsequence. The second part is to show that the limit is unique. 

We first state the following lemma for the sequel use. Let $j\ge 1$. Recall that $\lambda_j$ is the $j$-th eigenvalue of $\Delta$ on $K$ and $\Phi_j$ is the corresponding eigenfunction.
\begin{lem}\label{lem:L-infinity}
For any $j\ge 1$, we have
\[
\|\Phi_j\|_{L^\infty(K,\mu)} \le C \lambda_j^{d_h/(2d_w)}.
\]
\end{lem}
\begin{proof}
We use spectral theory (as in the proof of \cite[Lemma 3.4]{BC}).
Notice that $P_t\Phi_j=e^{-\lambda_j t}\Phi_j$.
Using the Cauchy-Schwartz inequality and \eqref{eq:subGauss-upper}, we obtain
for $\mu$-a.e. $x\in K$, 
\[
|\Phi_j(x)|=e^{\lambda_j t_0}|P_{t_0} \Phi_j(x)| \le e^{\lambda_j t_0}\left( \int_K p_{t_0}(x,y) ^2d\mu(y)\right)^{1/2}
\le Ct_0^{-d_h/(2d_w)} e^{\lambda_j t_0}.
\] 
In particular, taking $t_0=\lambda_j^{-1}$ leads to 
\[
|\Phi_j(x)|\le  C \lambda_j^{d_h/(2d_w)}, \quad\mu\text{-a.e. } x\in K.
\]
\end{proof}

\begin{prop}
The sequence $(\bar X_s^{\,m})_{m\ge 1}$ is tight in $H^{-\alpha}(K)$ for any $\alpha>2d_h/d_w$.
\end{prop}

\begin{proof}
We will first prove that for any $\varepsilon>0$, there exists  $R=R(\varepsilon)>0$ such that for all $m\ge 0$
\begin{equation}\label{eq:tightness}
\mathbb P(\|\bar X_s^{\,m}\|_{H^{-\alpha}(K)}^2> R)\le \varepsilon.
\end{equation}
Note that by Chebyshev's inequality, 
\[
\mathbb P(\|\bar X_s^{\,m}\|_{H^{-\alpha}(K)}^2> R)\le \frac1R \mathbb E(\|\bar X_s^{\,m}\|_{H^{-\alpha}(K)}^2).
\]
From Lemma \ref{lem:Sobolev dual},  we can write $\mathbb E(\|\bar X_s^{\,m}\|_{H^{-\alpha}(K)}^2)$ as
\[
\mathbb E\Bigg(\sum_{j=1}^{\infty} \lambda_j^{-\alpha} \left(\bar X_s^{\,m}(\Phi_j)\right)^2 \Bigg)
=\sum_{j=1}^{\infty}  \lambda_j^{-\alpha} \mathbb E\left( \left(X_s^{\,m}((\bar \Phi_j)_m)\right)^2 \right).
\]
Noticing that from \cite[Lemma 5.2]{FukushimaShima} one has  $\inf_m \lambda_1^m>0$, then applying \eqref{eq:DCov} gives
\[
\mathbb E\left( \left(X_s^{\,m}((\bar \Phi_j)_m)\right)^2 \right)=\frac{1}{a_m}\sum_{p\in V_m}\bar \Phi_j(p) (-\Delta_m)^{-2s}\bar \Phi_j(p)
\le (\lambda_1^m)^{-2s} \|\bar\Phi_j\|_{L^2(V_m,\mu_m)}^2.
\]
Observe that by Lemma \ref{lem:L-infinity} one has
\[ 
\|\bar\Phi_j\|_{L^2(V_m,\mu_m)}^2 \le 2\|\Phi_j\|_{L^\infty(K,\mu)}^2  \le C \lambda_j^{d_h/d_w}.
\] 
Besides, Weyl's eigenvalue asymptotics \eqref{Weyl_gasket} yields that  $\lambda_j\sim j^{d_w/d_h}$. Hence
\[
\mathbb E\left(\|\bar X_s^{\,m}\|_{H^{-\alpha}(K)}^2\right) \le C\sum_{j=1}^{\infty} \lambda_j^{-\alpha+d_h/d_w}
\le C\sum_{j=1}^{\infty} j^{\big(\frac{d_h}{d_w}-\alpha\big) \frac{d_w}{d_h}}.
\]
The above series is bounded if $1-\frac{\alpha d_w}{d_h}<-1$, i.e., $\alpha>2d_h/d_w$. Hence \eqref{eq:tightness} holds.

Now fix $\alpha>2d_h/d_w$. Then \eqref{eq:tightness} holds for any $\alpha'\in(2d_h/d_w,\alpha)$ and any $\varepsilon>0$. Equivalently,  there exists $R>0$ such that 
\[
\mathbb P\Big(\bar X_s^{\,m}\notin \overline{B_{-\alpha'}(0,R)}\Big)\le \varepsilon,
\]
where $\overline{B_{-\alpha'}(0,R)}$ denotes the closed ball with radius $R$ and center $0$ in $H^{-\alpha'}$. To conclude the proof, it suffices to show that $\overline{B_{-\alpha'}(0,R)}$ is compact in  $H^{-\alpha}$. Indeed, this can be seen from Rellich's theorem, i.e., the embedding $H^{\alpha} \hookrightarrow H^{\beta}$ is compact for $\beta<\alpha$, see the proof of \cite[Theorem 3.15]{CDH}.

\end{proof}

\begin{prop}
For any $f\in \mathcal S(K)$, one has $\bar X_s^{\,m}(f) \longrightarrow X_s(f)$ as $m\to \infty$.
\end{prop}
\begin{proof}
Recall that $\bar X_s^{\,m}(f)=X_s^m(\bar f_m)$. Since  fractional Gaussian fields are centered, it suffices to show that as $m\to \infty$,
\[
\mathbb E\left(\left(\bar X_s^{\,m}(f)\right)^2\right) \longrightarrow \int_K f(-\Delta)^{-2s}fd\mu.
\] 
We will use similar proof as \cite[Proposition 4.5]{CiprianiGinkel} for which Lemma \ref{holder regu}  is a crucial ingredient.

First observe that by \eqref{eq:DCov}, one has
\[
\mathbb E\left(\left(\bar X_s^{\,m}(f)\right)^2\right)=\mathbb E\left(\left(X_s^m(\bar f_m)\right)^2\right)
=\frac1{a_m}\sum_{p\in V_m}\bar f_m(p)(-\Delta_m)^{-2s}\bar f_m(p).
\]
Hence it remains to prove that 
\[
\frac1{a_m}\sum_{p\in V_m}\bar f_m(p)(-\Delta_m)^{-2s}\bar f_m(p) \longrightarrow \int_K f(-\Delta)^{-2s}fd\mu.
\]
Recall the convergence \eqref{eq:convergence} with the notation in Remark \ref{rem:discrete integral}, i.e., for $f_m=f|_{V_m}$, 
\[
\int_{V_m} f_m (-\Delta_m)^{-2s} f_m d\mu_m \longrightarrow \int_K f(-\Delta)^{-2s}fd\mu.
\] 
We thus need to show that 
\[
\int_{V_m} \bar f_m(-\Delta_m)^{-2s}\bar f_m d\mu_m- \int_{V_m} f_m(-\Delta_m)^{-2s} f_md\mu_m \longrightarrow 0.
\]
Indeed, the triangular inequality and Cauchy-Schwarz inequality yield
\begin{align*}
&\left|\int_{V_m} \bar f_m(-\Delta_m)^{-2s}\bar f_md\mu_m- \int_{V_m}f_m(-\Delta_m)^{-2s} f_md\mu_m\right|
\\ \le &
\int_{V_m} \left|\left(\bar f_m-f_m\right)(-\Delta_m)^{-2s} f_m\right|d\mu_m
+\int_{V_m} \left|\bar f_m (-\Delta_m)^{-2s} (f_m-\bar f_m)\right|d\mu_m
\\ \le &
\|\bar f_m-f_m\|_{L^2(V_m,\mu_m)} \left\|(-\Delta_m)^{-2s} f_m\right\|_{L^2(V_m,\mu_m)}
+\|\bar f_m\|_{L^2(V_m,\mu_m)} \left\|(-\Delta_m)^{-2s} (f_m-\bar f_m)\right\|_{L^2(V_m,\mu_m)}.
\end{align*}
One has then $\left\|(-\Delta_m)^{-2s} f_m\right\|_{L^2(V_m,\mu_m)}\le (\lambda_1^m)^{-2s}\|f_m\|_{L^2(V_m,\mu_m)}$ and 
\[
\left\|(-\Delta_m)^{-2s} (f_m-\bar f_m)\right\|_{L^2(\mu_m)} \le (\lambda_1^m)^{-2s} \|f_m-\bar f_m\|_{L^2(\mu_m)}.
\]
Note that from Lemma \ref{lem:measure}, $\|f_m\|_{L^2(V_m,\mu_m)}^2 \to \|f\|_{L^2(K,\mu)}^2$. Recall also $\inf_m \lambda_1^m>0$. It remains to show that  $\|f_m-\bar f_m\|_{L^2(V_m,\mu_m)} \to 0$. By Lemma \ref{holder regu},
\begin{align*}
\|f_m-\bar f_m\|_{L^2(V_m,\mu_m)}^2 
&=\frac{1}{a_m}\sum_{p\in V_m} |f_m(p)-\bar f_m(p)|^2
\\ & \le \frac{1}{a_m}\sum_{p\in V_m} \left(\fint_{C_p^m}|f_m(p)- f(x)| d\mu(x)\right)^2
\\ &\le C 2^{-2(m+1)(d_w-d_h)}\|\Delta f\|_{L^1(K,\mu)}^2.
\end{align*}
Now combining the above estimates and letting $m\to \infty$, we conclude the desired result.
\end{proof}

\begin{proof}[Proof of Theorem \ref{thm: CovergenceLaw}]
Since $(\bar X_s^{\,m})_{m\ge 1}$ is tight in $H^{-\alpha}(K)$ for any $\alpha>2d_h/d_w$, it is enough to show that every convergent subsequence $(\bar X_s^{\,m_k})_{k\ge 1}$ converges in law to $X_s$ in $H^{-\alpha}(K)$, that is, $\bar X_s^{\,m_k}(f)\to X_s(f)$ as $k\to \infty$ for all $f\in H^{\alpha}(K)$. 

Indeed, let $f\in H^{\alpha}(K)$, then there exists a sequence $(f_i)_{i\ge 1}\in \mathcal S$ such that $f_i\to f$ in $H^{\alpha}(K)$ and thus $(\bar{f_i})_{m_k} \to \bar f_{m_k}$ as $i\to \infty$. Therefore  $\bar X_s^{\,m_k}(f_i)$ and $X_s(f_i)$ converge to $\bar X_s^{\,m_k}(f)$ and $X_s(f)$ respectively as $i$ goes to infinity. Recall also $\bar X_s^{\,m_k}(f_i)\to \bar X_s(f_i)$ as $k\to \infty$. The triangle inequality thus concludes our proof.

\end{proof}
 
\bibliographystyle{abbrv}
\bibliography{bibfile}

\begin{thebibliography}{10}

\bibitem{AdlerTaylor}
R.~J. Adler and J.~E. Taylor.
\newblock {\em Random fields and geometry}.
\newblock Springer Monographs in Mathematics. Springer, New York, 2007.

\bibitem{BV3}
P.~Alonso-Ruiz, F.~Baudoin, L.~Chen, L.~Rogers, N.~Shanmugalingam, and
  A.~Teplyaev.
\newblock Besov class via heat semigroup on {D}irichlet spaces {III}: {BV}
  functions and sub-{G}aussian heat kernel estimates.
\newblock {\em Calc. Var. Partial Differential Equations}, 60(5):Paper No. 170,
  38, 2021.

\bibitem{Barlow}
M.~T. Barlow.
\newblock Diffusions on fractals.
\newblock In {\em Lectures on probability theory and statistics
  ({S}aint-{F}lour, 1995)}, volume 1690 of {\em Lecture Notes in Math.}, pages
  1--121. Springer, Berlin, 1998.

\bibitem{BarlowPerkins}
M.~T. Barlow and E.~A. Perkins.
\newblock Brownian motion on the {S}ierpi\'{n}ski gasket.
\newblock {\em Probab. Theory Related Fields}, 79(4):543--623, 1988.

\bibitem{BC}
F.~Baudoin and L.~Chen.
\newblock ${L}^p$-poincar\'e inequalities on nested fractals, 2020.

\bibitem{BaudoinLacaux}
F.~Baudoin and C.~Lacaux.
\newblock Fractional {G}aussian fields on the {S}ierpi\'{n}ski {G}asket and
  related fractals.
\newblock {\em J. Anal. Math.}, 146(2):719--739, 2022.

\bibitem{biskup}
M.~Biskup.
\newblock Extrema of the two-dimensional discrete {G}aussian free field.
\newblock In {\em Random graphs, phase transitions, and the {G}aussian free
  field}, volume 304 of {\em Springer Proc. Math. Stat.}, pages 163--407.
  Springer, Cham, [2020] \copyright 2020.

\bibitem{CDH}
A.~Cipriani, B.~Dan, and R.~S. Hazra.
\newblock The scaling limit of the membrane model.
\newblock {\em Ann. Probab.}, 47(6):3963--4001, 2019.

\bibitem{CiprianiGinkel}
A.~Cipriani and B.~van Ginkel.
\newblock The discrete {G}aussian free field on a compact manifold.
\newblock {\em Stochastic Process. Appl.}, 130(7):3943--3966, 2020.

\bibitem{FukushimaShima}
M.~Fukushima and T.~Shima.
\newblock On a spectral analysis for the {S}ierpi\'{n}ski gasket.
\newblock {\em Potential Anal.}, 1(1):1--35, 1992.

\bibitem{Kigami}
J.~Kigami.
\newblock {\em Analysis on fractals}, volume 143 of {\em Cambridge Tracts in
  Mathematics}.
\newblock Cambridge University Press, Cambridge, 2001.

\bibitem{Kigami2}
J.~Kigami.
\newblock Resistance forms, quasisymmetric maps and heat kernel estimates.
\newblock {\em Mem. Amer. Math. Soc.}, 216(1015):vi+132, 2012.

\bibitem{Kurtz}
T.~G. Kurtz.
\newblock Extensions of {T}rotter's operator semigroup approximation theorems.
\newblock {\em J. Functional Analysis}, 3:354--375, 1969.

\bibitem{LSSW}
A.~Lodhia, S.~Sheffield, X.~Sun, and S.~S. Watson.
\newblock Fractional {G}aussian fields: a survey.
\newblock {\em Probab. Surv.}, 13:1--56, 2016.

\bibitem{Meyer}
P.-A. Meyer.
\newblock Le th\'{e}or{\`e}me de continuit\'{e} de {P}. {L}\'{e}vy sur les
  espaces nucl\'{e}aires (d'apr{\`e}s {X}. {F}ernique).
\newblock In {\em S\'{e}minaire {B}ourbaki, {V}ol. 9}, pages Exp. No. 311,
  509--522. Soc. Math. France, Paris, 1995.

\bibitem{Pietsch}
A.~Pietsch.
\newblock \"{U}ber die {E}rzeugung von {$(F)$}-{R}\"{a}umen durch
  selbstadjungierte {O}peratoren.
\newblock {\em Math. Ann.}, 164:219--224, 1966.

\end{thebibliography}
\

\noindent
Fabrice Baudoin: \url{fabrice.baudoin@uconn.edu}\\
Department of Mathematics,
University of Connecticut,
Storrs, CT 06269

\

\noindent Li Chen: \url{lichen@lsu.edu}\\
Department of Mathematics, Louisiana State University, Baton Rouge, LA 70803

\end{document}